\documentclass[10pt, reqno, oneside, english]{smfart}

\usepackage{smfthm}

\usepackage[scale=0.73]{geometry}
\usepackage{amsmath, amsthm, amssymb,amsfonts}
\usepackage[utf8]{inputenc}
\usepackage[english]{babel}
\usepackage[T1]{fontenc}
\usepackage{url}
\usepackage{braket}
\usepackage{enumerate}

\usepackage[colorlinks=true, linktoc=page, citecolor=blue, linkcolor=blue, urlcolor=blue]{hyperref}

\numberwithin{equation}{section}

\newcommand{\R}{\mathbb{R}}
\newcommand{\C}{\mathbb{C}}

\newcommand{\N}{\mathbb{N}}

\newcommand{\g}{\mathfrak{g}}
\newcommand{\ka}{\mathfrak{k}}
\newcommand{\pe}{\mathfrak{p}}
\newcommand{\norme}[1]{\left\Vert #1\right\Vert}

\newcommand{\inde}{\text{Ind}}
\newcommand{\ind}{\text{\emph{Ind}}}

\newcommand{\SL}{\mathrm{SL}}
\newcommand{\SO}{\mathrm{SO}}

\title{Continuity of the Mackey-Higson bijection}
\author{Alexandre Afgoustidis}
\address{CEREMADE, Université Paris-Dauphine, PSL, 75016 Paris, France}
\curraddr{\emph{Current Affiliation:} CNRS \& Institut Élie Cartan de Lorraine, UMR 7502}
\email{alexandre.afgoustidis@math.cnrs.fr}
\author{Anne-Marie Aubert}
\address{Institut de Math\'ematiques de Jussieu – Paris Rive Gauche, CNRS, Sorbonne Universit\'e, Universit\'e de Paris, 75005 Paris, France}
\email{anne-marie.aubert@imj-prg.fr}

\begin{document}

\begin{abstract}  {When $G$ is a real reductive group and $G_0$ is its Cartan motion group, the Mackey-Higson bijection is a natural one-to-one correspondence between all irreducible tempered representations of $G$ and all irreducible unitary representations of $G_0$. In this short note, we collect some known facts about the topology of the tempered dual $\widetilde{G}$ and that of the unitary dual $\widehat{G_0}$, then verify that the Mackey-Higson bijection $\widetilde{G} \to \widehat{G_0}$ is continuous.}\end{abstract}

\maketitle
\section{Introduction}\label{sec:intro}

Let $G$ be the group of real points of a connected reductive algebraic group defined over $\R$. Given a maximal compact subgroup $K$, form the Cartan motion group $G_0$ attached to $G$ and $K$: if $\g=\ka \oplus \pe$ is the Cartan decomposition of the Lie algebra of $G$, then $G_0$ is the semidirect product $K \ltimes \pe$. Write $\widetilde{G}$ for the tempered dual of $G$ and $\widehat{G_0}$ for the unitary dual of $G_0$.

Recent years have brought proof that the parameters necessary to describe $\widetilde{G}$ and $\widehat{G_0}$ are identical: there is a natural, but non-trivial, one-to-one correspondence between $\widetilde{G}$ and $\widehat{G_0}$  \cite{MackeyConjecture, Higson2008, HigsonComplex, AAMackey}. We will refer to this correspondence as the \emph{Mackey-Higson bijection}.

When equipped with the Fell topology, neither $\widetilde{G}$ nor $\widehat{G_0}$ is usually Hausdorff. But it has been known for a while that the existence of a topologically well-behaved bijection between $\widetilde{G}$ and $\widehat{G_0}$ could be of particular significance for certain aspects of representation theory, especially in relation with the study of group $C^\ast$-algebras. Alain Connes and Nigel Higson pointed out in the late 1980s $-$ long before the actual construction of the Mackey-Higson bijection $-$ that the \emph{Baum-Connes-Kasparov isomorphism} for the $K$-theory of the reduced $C^\ast$-algebra of $G$ can be viewed as a statement that $\widetilde{G}$ and $\widehat{G_0}$, although not homeomorphic, always have in a precise sense the same $K$-theory (see \cite{BaumConnesHigson}).
 
The Mackey-Higson bijection \emph{does} have pleasant topological properties: building on $C^\ast$-algebraic methods due to Nigel Higson \cite{Higson2008} in the complex case, one of us showed in \cite{AAKasparov} that the bijection is a \emph{piecewise homeomorphism}. For real groups, the homeomorphic pieces are defined through David Vogan's theory of lowest $K$-types. The pieces are stitched together differently in both duals; but taking $K$-theory somehow blurs out that fact: the Connes-Kasparov isomorphism can actually be obtained in a rather elementary way from the topological properties of the Mackey-Higson bijection (this is the main result of \cite{AAKasparov}). 

In this short paper, we complete this topological information by proving that the Mackey-Higson bijection maps $\widetilde{G}$ \emph{continuously} onto  $\widehat{G_0}$, although it is never a homeomorphism (except in trivial cases). Together with the results of \cite{AAKasparov}, we obtain the following picture for its topological properties: 

\begin{theo}\label{main}
\begin{enumerate}[(a)]
\item \emph{(Proposition \ref{disco} and Theorem \ref{continuite} below.)}  {The Mackey-Higson bijection $\mathcal{M}: \widetilde{G} \to \widehat{G_0}$ is continuous, but $\mathcal{M}^{-1}$ is not continuous unless $G$ and $G_0$ are isomorphic.}
\item (\cite[Corollary 3.7]{AAKasparov}) {Let $\mathcal{C}$ be a finite subset of the unitary dual $\widehat{K}$ of $K$, and let $\widetilde{G}[\mathcal{C}]$ (resp. $\widehat{G_0}[\mathcal{C}]$) be the (possibly empty) subset of $\widetilde{G}$ (resp. $\widehat{G_0}$) consisting of those representations whose set of lowest $K$-types is exactly $\mathcal{C}$. The bijection $\mathcal{M}$ induces a homeomorphism between $\widetilde{G}[\mathcal{C}]$ and $\widehat{G_0}[\mathcal{C}]$.}
\end{enumerate}
\end{theo}

Let us  further comment on property (a). The construction of the Mackey-Higson bijection is easier to write down in the $\widehat{G_0} \to \widetilde{G}$ direction; yet it is the reverse map, from $\widetilde{G}$ to $\widehat{G_0}$, which turns out to be continuous. This is only natural given the link with noncommutative geometry, particularly the recent work of Nigel Higson and Angel Rom\'an on reduced $C^\ast$-algebras for complex groups \cite{HigsonRoman}. They proved the existence of a nontrivial continuous embedding $C^\ast_r(G_0) \to C^\ast_r(G)$, and characterized the Mackey bijection in terms of this embedding $-$ in the particular case of complex groups, the existence of the embedding can be viewed as morally ``dual'' to our continuity result on the reduced duals.

We should mention that most of the available information on the Mackey bijection has been obtained through the existence of a continuous family $(G_t)_{t \in [0,1]}$ of groups interpolating between $G$ and $G_0$: see \cite{AAContractions, AAKasparov, HigsonRoman, Subag}. By contrast, our method here is direct:  we use the fine details of the algebraic construction in \cite{AAMackey}, together with various known results on the topology of representation spaces. The properties we shall need are essentially algebraic, often boiling down to a careful analysis of lowest $K$-types, and do not rely on deformation theory.

Our proof, in \S \ref{sec:resultat}, will be obtained by focusing on the restriction of the Mackey-Higson bijection to each connected component of $\widetilde{G}$. The Fell topology on $\widetilde{G}$ has been known quite precisely since the 1980s \cite{Delorme_topo, VoganIsolated}, and the topology of $\widehat{G_0}$  has been described in 1968 \cite{Baggett};  we will use these descriptions and recall the necessary details in \S \ref{sec:temp} and \S \ref{sec:cartan}.

We close this Introduction with a few remarks on the possibility that analogues of properties (a) and (b) above may hold true for the admissible dual. (We thank the referee for suggesting and clarifying this.) Consider the admissible duals $\Pi_{\mathrm{adm}}(G)$ and $\Pi_{\mathrm{adm}}(G_0)$. The Mackey-Higson bijection extends to a bijection $\mathcal{M}_{\mathrm{adm}}: \Pi_{\mathrm{adm}}(G) \to \Pi_{\mathrm{adm}}(G_0)$: see \cite[\S 5]{AAMackey}. 

Furthermore, both admissible duals are stratified according to lowest $K$-type theory. Given a class $\sigma \in  \widehat{K}$, we can consider the subsets $\Pi_{\sigma}(G) \subset \Pi_{\mathrm{adm}}(G)$ and $\Pi_{\sigma}(G_0) \subset \Pi_{\mathrm{adm}}(G_0)$ consisting of those representations that admit $\sigma$ as a lowest $K$-type. These strata do not define a partition of the admissible duals, in contrast to the subsets $\widetilde{G}[\mathcal{C}]$ of property (b): strata attached to different representations $\sigma$ may intersect nontrivially. But each stratum is an affine algebraic variety (see \cite[Remark 5.5]{Subag}), and the intersection of each stratum $\Pi_{\mathrm{adm}}(G)$ with the tempered dual $\widetilde{G}$ is a finite disjoint union of subsets of the form $\widetilde{G}[\mathcal{C}]$. 

Now, when $G$ is a complex semisimple group, Nigel Higson proved in \cite{HigsonComplex} that for each $\sigma \in \widehat{K}$, the map $\mathcal{M}_{\mathrm{adm}}:  \Pi_{\mathrm{adm}}(G) \to \Pi_{\mathrm{adm}}(G_0)$ induces an isomorphism of algebraic varieties between $\Pi_{\sigma}(G) $ and $\Pi_{\sigma}(G_0)$. Eyal Subag conjectured in \cite[\S 5]{Subag} that the same property holds for arbitrary $G$. He verified this (and much more) in \cite[\S 6]{Subag} for $G=\mathrm{SL}(2,\R)$ using the algebraic framework of \cite{BHS1, BHS2}. 

In view of these results, it it tempting to wonder about the topological properties of the extension $\mathcal{M}_{\mathrm{adm}}: \Pi_{\mathrm{adm}}(G) \to \Pi_{\mathrm{adm}}(G_0)$, and to speculate that it may be continuous (in the Fell topologies described e.g. in  \cite{Rader, Fell_adm, Pandzic}), but not a homeomorphism except in trivial cases. Furthermore, it seems that Subag's conjecture may be the final word on the algebraic analogue of property (b) above. 

However, the Fell topology on  $\Pi_{\mathrm{adm}}(G)$ is more subtle than that of the tempered dual  (see \cite{Pandzic}). Because non-associated principal series representations can have a common subquotient, the connected components of $\Pi_{\mathrm{adm}}(G)$ cannot be described by simple results parallel to those of \S \ref{composantes_connexes}. Therefore, extending our method to cover the case of the admissible duals would necessitate results on the Fell topology of $\Pi_{\mathrm{adm}}(G)$ which we have been unable to locate in the literature, or to prove ourselves. As a result, we shall remain entirely confined within the tempered duals here.

\begin{enonce*}[remark]{Acknowledgments}  We thank Nigel Higson for dissipating an early misconception about the topology of $\widehat{G_0}$, and Maarten Solleveld for correcting a mistake in the first version of Proposition \ref{disco}.\end{enonce*}
\section{Topology of the tempered dual} \label{sec:temp}

\subsection{Harish-Chandra decomposition} 

Suppose $G$ is the group of real points of a connected reductive algebraic group defined over $\R$. Fix a Haar measure on $G$,  form the reduced  $C^\star$-algebra  $C^\star_r(G)$ and the category
\[ \mathcal{M}^t(G) = \text{ category of continuous  nondegenerate }  C^\star_r(G)\text{-modules}\]
of tempered representations of $G$. 

Let us first recall how the general shape of Harish-Chandra's Plancherel formula yields an infinite direct product decomposition of $\mathcal{M}^t(G)$. The presentation is modeled on the $p$-adic case, more precisely, on Schneider and Zink's tempered version of the Bernstein decomposition of the category of admissible representations \cite{KZ} (see also \cite[\S III.7]{Waldspurger}). The results of this paragraph are true exactly as stated if $F$ is any local field and $G$ is the group of $F$-points of a connected reductive algebraic group defined over $F$.\\

Let us call \emph{discrete pair} any pair $(L, \sigma)$ in which 
\begin{itemize}
\item[$\bullet$]  $L$ is a Levi subgroup of $G$
\item[$\bullet$]  $\sigma \in \widetilde{L}$ is the equivalence class for a {tempered} irreducible representation of $L$ that is square-integrable-modulo-center.
\end{itemize}

The group $G$ acts on the set of discrete pairs: if $(L, \sigma)$ is a discrete pair and $\bar{\sigma}$ is a tempered irreducible representation of $L$ with class $\sigma$, then for $g \in G$, the discrete pair $g \cdot (L, \sigma)$ is the  pair $(gLg^{-1}, \text{class of } \bar{\sigma}(g^{-1} \cdot g))$. We write  $\Omega^t(G)$ for the space of orbits.

For every Levi subgroup $L$ of $G$, let us call \emph{unramified unitary  character} of $L$ any unitary character of $L$ that is trivial on every compact subgroup of $L$; we will write  $\mathcal{X}_u(L)$ for the set of unramified unitary characters of $L$. 

Let us now fix a discrete pair $(L, \sigma)$. The map
\begin{align*} \Phi_{L, \sigma}:\ \mathcal{X}_u(L) & \to \Omega^t(G) \\ \chi & \mapsto \text{ orbit of } (L, \sigma \otimes \chi) \end{align*}
sends $\mathcal{X}_u(L)$ to a subset $\Theta_{L, \sigma}=\Phi_{L, \sigma}\left( \mathcal{X}_u(L) \right)$ of $\Omega^t(G)$. Inequivalent pairs $(L, \sigma)$ map to disjoint subsets of $\Omega^t(G)$. We write $\mathcal{B}^t(G) = \left\{ \Theta_{L,\sigma}, \ [(L, \sigma)] \in \Omega^t(G)\right\}$ for the set of blocks.

Suppose $\pi$ is an irreducible tempered representation of $G$. Then there exists a discrete pair $(L, \sigma)$ and  a parabolic subgroup $P=LN$ of $G$ with the property that $\pi$ is equivalent with one of the irreducible factors of $\inde_{LN}^G(\sigma)$. The pair $(L, \sigma)$ determines an element $\Theta_{L,\sigma}$  of $\mathcal{B}^t(G)$ which depends only on $\pi$, not on the choice of $(L, \sigma)$: we may and will call it the \emph{discrete support} of $\pi$. For all this, original sources include \cite{Trombi, KZ_PNAS}; a convenient reference is \cite[\S 5-6]{CCH}.

For every block $\Theta \in \mathcal{B}^t(G)$, we write $\mathcal{M}^t(\Theta)$ for the category of continuous nondegenerate $C^\star_r(G)$-modules whose irreducible factors all have discrete support $\Theta$. Harish-Chandra's work induces a direct sum decomposition 
\begin{equation} \label{CCH} C^\star_r(G) = \sum \limits_{\Theta \in \mathcal{B}^t(G)}  C^\star_r(G)_{\Theta},\end{equation}
where the spectrum of a given component $C^\star_r(G)_{\Theta}$ is
\begin{equation} \label{bloc} \widetilde{G}_{\Theta}= \left\{ \pi \in \widetilde{G} \quad \text{whose discrete support is $\Theta$} \right\}; \end{equation}
see \cite[Proposition 5.17 and Theorem 6.8]{CCH}. This yields a partition 
\[ \widetilde{G} = \bigsqcup \limits_{\Theta \in  \mathcal{B}^t(G)}  \widetilde{G}_{\Theta}\]
of the tempered dual into disjoint subsets, and a direct product decomposition
\[ \mathcal{M}^t(G) = \prod \limits_{\Theta \in \mathcal{B}^t(G)}  \mathcal{M}^t(\Theta)\]
of the category of tempered representations. We refer to \cite{KZ} for the parallel with the non-archimedean case. 

\subsection{Connected components of the tempered dual $\widetilde{G}$} \label{composantes_connexes} The following remark is important for what follows. Although it is well-known (see for instance \cite[théorème 2.6(ii)]{Delorme_topo}), we will sketch a proof. 

\begin{prop} \label{connexes} The connected components of $\widetilde{G}$ are the  $\widetilde{G}_{\Theta}$, $\Theta \in \mathcal{B}^t(G)$. \end{prop}

\begin{proof}  We note first that for every $\Theta$, the subset $\widetilde{G}_{\Theta}$ of $\widetilde{G}$ is closed: the decomposition in \eqref{CCH} identifies  $\widetilde{G}_{\Theta}$ with the set of irreducible representations of $C^\star_r(G)$ that vanish on the ideal $J_\Theta=\sum \limits_{{\Theta}' \neq \Theta} C^\star_r(G)_{{\Theta}'}$. That is indeed a closed subset \cite[\S 3.2]{Dixmier}.

We now check that each $\widetilde{G}_{\Theta}$ is connected. To that end, we will fix an element  $\pi_\star$ of $\widetilde{G}_{\Theta}$ and prove that every element in  $\widetilde{G}_{\Theta}$ necessary lies in the same connected component of $\widetilde{G}$ as $\pi_\star$. 

Fix a component $\Theta \in \mathcal{B}^t(G)$. Among the discrete pairs $(L, \sigma)$ with equivalence class $\Theta$, there is one that has the additional property that if $L=MA$ is the Langlands decomposition of $L$ (where $M$ is the subgroup of $L$ generated by all compact subgroups of $L$, and $A$ is the split component), then the restriction $\sigma_{|A}$ is trivial. We will assume that our pair has that property. Writing $\mathfrak{a}$ for the Lie algebra of $A$ and $\mathfrak{a}^\star$ for its vector space dual, recall that an element of $\mathfrak{a}^\star$ is called $\mathfrak{a}$-\emph{regular} when its scalar product with every positive root of  $(\mathfrak{g}_\C, \mathfrak{a}_\C)$ is nonzero. Fix a parabolic subgroup $P=LN$ with Levi factor $L$; this determines an ordering of $\mathfrak{a}^\star$. Fix then an element $\nu_\star$ in $\mathfrak{a}^\star$ whose scalar product with every positive root of  $(\mathfrak{g}_\C, \mathfrak{a}_\C)$ is positive (and nonzero), then define $\pi_\star = \inde_{P}^G(\sigma \otimes e^{i\nu_\star})$. Since $\nu_\star$ is $\mathfrak{a}$-{regular}, that representation is irreducible (see e.g. \cite[Theorem 14.93]{Knapp}), and it lies~in~$\widetilde{G}_{\Theta}$.

For every representation $\pi$ in $\widetilde{G}_{\Theta}$, there exists an element $\nu$ in $\mathfrak{a}^\star$ which has the property that the scalar product of $\nu$ with every  positive root of  $(\mathfrak{g}_\C, \mathfrak{a}_\C)$ is nonnegative and that $\pi$ is equivalent with one of the irreducible constituents of $\inde_{P}^G(\sigma \otimes e^{i\nu})$.

For every $t$ in $[0,1]$, define $\nu_t = t \nu + (1-t)\nu_\star$ and $\pi_t = \inde_P^G(\sigma \otimes e^{i \nu_t})$. For  $t>0$, the element $\nu_t$ is $\mathfrak{a}$-regular because it has a positive scalar product with every  positive root of  $(\mathfrak{g}_\C, \mathfrak{a}_\C)$; thus  $\pi_t$ is irreducible for $t>0$. %

Using the ``compact picture''  for induced representations, together with the routine criteria for the continuity of parameter integrals, we can exhibit\footnote{Fix a Hilbert space $V^\sigma$ for $\sigma$; then for each $t \in [0,1]$, the representation $\pi_t$ can be realized on the Hilbert space $\mathcal{H}~=~\left\{ f: K \overset{L^2}{\longrightarrow} V^\sigma \ : \ \forall u \in (K \cap M), \ f(ku) = \sigma(u) f(k)\right\}$. Suppose $\xi$ and $\eta$ are elements of $\mathcal{H}$; then the matrix element $c_{\xi, \eta}: g \mapsto \langle \xi, \pi_t(g)\eta\rangle$ can be expressed using the Iwasawa decomposition $G=KMAN$. If an element of $G$ reads $g = \kappa(g) \mu(g) e^{H(g)} \nu(g)$, where $( \kappa(g), \mu(g), H(g),\nu(g)) \in K\times  (M \cap \exp(\pe)) \times \mathfrak{a} \times N$, then we have $c_{\xi, \eta}(g) = \int_{K} e^{\langle-i  \nu_t - \rho, \ H(u^{-1}g)\rangle}  f(g,u) du$, where $f(g,u) =  \langle \xi(u),  \sigma(\mu(u^{-1} g)) \eta(\kappa(u^{-1}g)\rangle_{V^\sigma}$ and $\rho$ is the usual half-sum of positive roots. See for instance \cite[\S 2.2]{vandenban}. The claim easily follows.} every matrix element of $\pi$ as a  limit (in the sense of uniform convergence on compact subset of $G$) of a family of matrix elements of $(\pi_t)_{t>0}$. This proves that $\pi$ is in the closure of the family   $(\pi_t)_{t>0}$, and thus $\pi=\pi_0$ and $\pi_\star=\pi_1$ must lie in the same connected component of $\widetilde{G}$.
\end{proof}

\subsection{Vogan's lowest-$K$-type picture for the decomposition} 

We now fix a maximal compact subgroup $K$ in $G$ and recall Vogan's description of the ``blocks'' $\widetilde{G}_{\Theta}$  in terms  of associate classes of lowest $K$-types (see \cite{Vogan85}).

Fix $\Theta$ in $\widetilde{G}_\Theta$. Let $(L, \sigma)$ be a discrete pair with equivalence class $\Theta$ and the additional property that the restriction $\sigma_{|A}$ to the split component $A$ of $L$ is trivial (see the proof of Proposition \ref{connexes}). Fix a parabolic subgroup $P$ of $G$ with Levi factor $L$, then write $\mathcal{C}_{\Theta}$ for the set of lowest $K$-types of the representation $\inde_{P}^G(\sigma)$. The notation is coherent, because all choices of $(L, \sigma)$ with equivalence class $\Theta$, provided $\sigma$ is trivial on the split component, lead to the same set of lowest $K$-types.

\begin{prop}[Vogan] \label{vogan} Let $\pi$ be an irreducible tempered representation of $G$. Then $\pi$ lies in the component $\widetilde{G}_{\Theta}$ if and only if at least one of its lowest $K$-types lies in $\mathcal{C}_{\Theta}$. When that is the case, all the lowest $K$-types of $\pi$  lie in $\mathcal{C}_{\Theta}$.   \end{prop}

See  \cite[Introduction]{VoganAnnals} for connected $G$, and \cite[\S 4]{VoganSpeh} for a class of groups that includes the current one. \qed

\subsection{Closure of a subset of $\widetilde{G}$}

We now formulate a version of the available results which yield a complete description of the topology of $\widetilde{G}$; let us mention contributions of Delorme \cite{Delorme_topo} and Mili\v{c}i\'c (see Vogan \cite{VoganIsolated}).

\begin{enonce}[remark]{Notations}  \label{notations} Fix a component  $\widetilde{G}_\Theta$ in the tempered dual; assume  $\widetilde{G}_\Theta$ is not a single point. Consider
\begin{align*}
 (L, \sigma) :& \text{ a discrete pair corresponding to $\Theta$ ;}\\
  L=MA :& \text{ the Langlands decomposition of $L$. }\\ &\text{We may and will assume that the restriction $\sigma_{|A}$ is trivial  } \\ &\text{and will identify $\sigma$ with its restriction to $M$ (a genuine discrete series representation).}\\
  P=MAN : &  \text{ a parabolic subgroup with Levi factor  $L$ }\\
\Delta^+ : &  \text{ the positive root system for $(\g_\C, \mathfrak{a}_\C)$ that corresponds to $N$ ;}\\
 \mathfrak{a}^{+} &= \left\{ \nu \in \mathfrak{a}^\star, \ \forall \alpha \in \Delta^+, \ \langle \alpha, \nu \rangle \geq 0\right\}. \text{ It is a closed cone in $\mathfrak{a}^\star$.}\\
{\mathfrak{a}}^\star\! /W &= \text{ the quotient of ${\mathfrak{a}}^\star$ by the Weyl group $W=W(\g_\C, \mathfrak{a}_\C)$. }\\ 
& \text{Every class in ${\mathfrak{a}}^\star\! /W$ has a finite (nonzero) number of representatives in  $\mathfrak{a}^+$.}
\end{align*}

For every irreducible tempered representation $\pi$ in the component $\widetilde{G}_\Theta$, there is a \emph{unique} $\nu=\nu(\pi)$ in ${\mathfrak{a}}^\star\! /W$ with the property that $\pi$ occurs in $\inde_{P}^G(\sigma \otimes e^{i\nu})$. 
\end{enonce} 

\begin{prop} Suppose $B$ is a subset of $\widetilde{G}_{\Theta}$. Form the set  $\mathcal{V} = \left\{ \nu(\pi), \ \pi \in B\right\}$ of continuous parameters for representations in  $B$, and its closure $\bar{\mathcal{V}} $ in ${\mathfrak{a}}^\star\! /W$. The closure of  $B$ in $\widetilde{G}$ consists of those  representations $\pi$ that belong to  $\widetilde{G}_\Theta$ and satisfy $\nu(\pi) \in \bar{\mathcal{V}}$. 

 \end{prop}
 
  \begin{proof} Combine \cite[Théorème 2.6]{Delorme_topo} or \cite[Theorem 3]{VoganIsolated} with the information on connected components from \S \ref{composantes_connexes}.  \end{proof}

To obtain the closure of $B$ in $\widetilde{G}$, we can thus form the set of representations $\inde_{P}^G(\sigma \otimes \chi)$ for $\chi$ in the closure of the set of continuous parameters for members of $B$, then consider all irreducible factors in these.

\begin{exem} \label{ex_lim} Suppose  $G =\SL(2, \R)$,  $MAN$  is the upper triangular subgroup  (recall that $M = \{\pm I_2\}$ and $A$ is the subgroup of diagonal $2 \times 2$ matrices with positive entries and determinant one), and $\sigma$ is the one non-trivial representation of $M$. The discrete pair $(MA, \sigma)$ determines a component  $\Theta$ in $\widetilde{G}$. We can identify the unitary unramified characters of $MA$ with the unitary characters of $A$. Now, consider the set $B$ of representations of the form $\inde_{MAN}^G(\sigma \otimes \chi)$, where $\chi$ is a \emph{non-trivial} unitary character of $A$. All representations in $B$ are irreducible and tempered. To obtain the closure of $B$ in $\widetilde{G}$, we need to add to $B$ the two ``limits of discrete series'' whose direct sum is the (reducible) representation $\inde_{MAN}^G(\sigma \otimes e^{i0})$. \end{exem}

We shall need a consequence of the above results, taken from \cite[Theorem 3]{VoganIsolated}, that says what happens when one considers the continuous parameters for a convergent sequence of irreducible tempered representations. 

\begin{coro} \label{conv_param_cont} Suppose $\pi$ is an irreducible tempered representation of $G$ in the component $\widetilde{G}_\Theta$, and  $(\pi_n)_{n \in \N}$ is a sequence of irreducible tempered representations that admits $\pi$ as a limit point. There exists a subsequence $(\pi_{n_j})_{j\in \N}$ of representations that all lie in  $\widetilde{G}_{\Theta}$ and have the property that as $j$ goes to infinity, $\nu(\pi_{n_j}) $ goes to $\nu(\pi)$ in ${\mathfrak{a}}^\star\! /W$.
 \end{coro}

\section{Topology of the motion group dual and remarks on the Mackey-Higson bijection} \label{sec:cartan}

\subsection{Mackey parameters for $\widehat{G_0}$ and $\widetilde{G}$} \label{param_mackey}

Suppose $\sigma$ is a unitary irreducible representation of the motion group $G_0=K \ltimes \pe$ of \S \ref{sec:intro}. Given a pair  $(\chi, \mu)$ in which $\chi$ is an element of $\mathfrak{p}^\star$ and $\mu$ is an irreducible representation of the stabilizer $K_\chi$ of $\chi$ in $K$, we shall say that $(\chi, \mu)$ is a \emph{Mackey parameter for $\sigma$} when $\sigma$ is equivalent, as a representation of $G_0$, with $\inde_{K_\chi \ltimes \pe}^{G_0}(\mu \otimes e^{i\chi})$. Every unitary irreducible representation $\sigma$ of $G_0$ admits at least one Mackey parameter, and two different Mackey parameters for a given $\sigma$ must be conjugate under $K$.

It is perhaps a bit less natural to write down the definition of the Mackey-Higson bijection $\mathcal{M}: \widetilde{G} \to \widehat{G_0}$ rather than that of its inverse, as in \cite{AAMackey}; yet the information we shall need can be phrased as follows. 

If $\pi$ is an irreducible tempered representation of the reductive group $G$, we will say that $(\chi, \mu)$ is a Mackey parameter for $\pi$ when it is a Mackey parameter for the representation $\mathcal{M}(\pi)$ of $G_0$. 

Crucial to the next statement are the notion of tempered representation with \emph{real infinitesimal character}  (see for instance \cite[\S 3.1]{AAMackey} for the useful details), and, when $A$ is the split component of a Levi subgroup of $G$, the notion of \emph{$\mathfrak{a}$-regular} element in $\mathfrak{a}^\star$ (see \cite[\S 3.2]{AAMackey}).

\begin{lemm} \label{lemme:mackey} Let $\pi$ be an irreducible tempered representation of $G$. The pair $(\chi, \mu)$ is a Mackey parameter for $\pi$ if and only there is an equivalence
\[  \pi \simeq \ind_{P}^{G}\left( \sigma \otimes e^{i\chi}\right),\]
where the data $(P, \sigma)$ satisfy the following conditions:
\begin{itemize}
\item[$\bullet$] $P=LN = MAN$ is a parabolic subgroup of $G$, 
\item[$\bullet$] $\chi$ is an  \emph{$\mathfrak{a}$-regular} element of $\mathfrak{a}^\star$,
\item[$\bullet$]  $\sigma$ is a tempered representation of $L$ that has real infinitesimal character and whose restriction to $K \cap L = K_\chi$ admits $\mu$ as its lowest $(K \cap L)$-type.
\end{itemize}
\end{lemm}

 The above statement depends on the fact $K \cap L = K_\chi$ whenever $\chi$ is $\mathfrak{a}$-regular, and on a theorem of Vogan (see \cite[Theorem 1.2]{VoganBranching}) that says that given $L$ and $\mu \in \widehat{K \cap L}$, there exists a (unique) $\sigma \in \widetilde{L}$ that satisfies the condition in the theorem. The fact that the statement is correct is a reformulation of parts of the construction of the correspondence in \cite[\S 3]{AAMackey}. 
\qed

\subsection{Discontinuity of $\mathcal{M}^{-1}: {\widehat{G_0} \to \widetilde{G}}$} Before we consider  the topology of $\widehat{G_0}$ in more detail, let us remark that we already know enough about $\widetilde{G}$ to verify that the Mackey bijection cannot be a homeomorphism.

\begin{prop} \label{disco} The bijection $\mathcal{M}^{-1}:{\widehat{G_0} \to \widetilde{G}}$ is never continuous, unless $G$ and $G_0$ are isomorphic. \end{prop}
An isomorphism between $G$ and $G_0$ exists if and only if $G$ is a direct product of a compact group and an abelian vector group, in other words (by \cite{KnappBeyond}, Proposition 7.27) when the derived group of $G$ is compact. When $G$ and $G_0$ are not isomorphic, we shall see, for instance, that the trivial representation of $G_0$ is always a discontinuity point of $\mathcal{M}^{-1}$.

\begin{proof} Suppose $G$ and $G_0$ are not isomorphic. Choose a  \emph{regular} element $\chi$ in $\pe^\star$ (see \cite[p.555]{KnappBeyond}, for the existence of such an element in our case). Consider the induced representation $\sigma=\inde_{M \ltimes \pe}^{G_0}(\mathbf{1} \otimes e^{i\chi})$, where $M=Z_K(\chi)$ is the centralizer of $\chi$ in $\pe^\star$.

For every $\alpha>0$, let us consider the representation $\sigma_{\alpha}=\inde_{M \ltimes \pe}^{G_0}(\mathbf{1} \otimes e^{i\frac{\chi}{\alpha}})$, and observe what happens as $\alpha$ goes to infinity. If  $\widehat{K}_M$ is the set of those $\lambda \in \widehat{K}$ that occur in $\mathbf{L}^2(K/M)$, and if we identify the elements of $\widehat{K}_M$ with representations of $G_0$ in which $\pe$ acts trivially, then we will check that  ${B}= \left\{\sigma_\alpha, \alpha \in [1, +\infty[\right\} \cup \widehat{K}_M$ is a  \emph{connected} subset of $\widehat{G_0}$.

Recall that we can realize $\sigma_\alpha$ as a representation that acts on $\mathbf{L}^2(K/M)$, in which $K$ acts through the left regular representation $\mathcal{L}$ on $\mathbf{L}^2(K/M)$, and $v \in \pe$ acts through $f \mapsto \left[ u \mapsto e^{i \langle{\chi,\text{Ad}(u^{-1})v\rangle}}f(u)\right]$. 

Suppose $\xi, \xi'$ are vectors in $\mathbf{L}^2(K/M)$ that both lie in the isotypical subspace of  $\mathbf{L}^2(K/M)$ for a class $\lambda \in \widehat{K}$. Then the matrix element of $\sigma_\alpha$ attached to $\xi$ and $\xi'$ $-$ a complex-valued function on $K \times \pe$ $-$ reads  $(k,v) \mapsto \int_{K} e^{i \langle \frac{\chi}{\alpha}, (\text{Ad}(u^{-1}) v)\rangle} \xi(u) \xi'(k^{-1}u) du$, where we viewed $\xi$ and $\xi'$ as right-$M$-invariant functions on $K$.

When $\alpha$ goes to infinity, that map converges (uniformly on compact subsets of  $K \times \pe$) to the map $(k,v) \mapsto \langle \xi, \mathcal{L}(k) \xi'\rangle_{\mathbf{L}^2(K/M)}$, which is a matrix element for the representation of  $G_0$ that extends $\lambda \in \widehat{K}_M$. This proves that ${B}$ is connected.

Now, the image of $B$ under $\mathcal{M}_{\widehat{G_0} \to \widetilde{G}}$ is not connected: the image of $\left\{\sigma_\alpha, \alpha \in [1, +\infty[\right\} \cup \left\{ \text{triv}_{\widehat{K}}\right\}$ is contained in the \emph{spherical principal series}, which is by itself a connected component\footnote{It is the connected component associated with a discrete pair of the form  $(L_{\min}, \mathbf{1})$ where $L_{\min}$ is the Levi factor for a Borel subgroup of $G$ and $\mathbf{1}$ is the trivial representation of $L_{\min}$. Every representation in the spherical principal series contains the trivial $K$-type.} of $\widetilde{G}$. By contrast, the nontrivial elements of $\widehat{K}_M$ map to irreducible representations that do not contain the trivial $K$-type (since the Mackey-Higson bijection preserves lowest $K$-types), and are thus contained in other connected components of  $\widetilde{G}$. Therefore the bijection $\mathcal{M}^{-1}:{\widehat{G_0} \to \widetilde{G}}$  sends the connected set $B$ to a disconnected subset of $\widetilde{G}$, and we can conclude that $\mathcal{M}^{-1}$ is not continuous.
 \end{proof}

\subsection{Baggett's description of the topology of $\widehat{G_0}$  }

\begin{theo} \label{Baggett} Let $(\sigma_n)_{n \in \N}$ be a sequence of unitary irreducible representations of $G_0$, and  $\sigma_\infty$ be a unitary irreducible representation of  $G_0$. For every $n$ in $\N  \cup\{ \infty\}$, fix a Mackey parameter $(\chi_n, \mu_n)$ for $\sigma_n$. 

The representation $\sigma_\infty$ is a limit point of $(\sigma_n)_{n \in \N}$ in $\widehat{G_0}$ if and only if $(\sigma_n)_{n \in \N}$ admits a subsequence  $(\sigma_{n_j})_{j \in \N}$ that satisfies the following conditions :

\begin{enumerate}[(i)]
\item as $j$ goes to infinity, $\chi_{n_j}$ goes to $\chi_\infty$ in $\pe^\star$;
\item the sequence $\left((K_{\chi_{n_j}}, \mu_{n_j})\right)_{j \in \N}$ is actually a constant $(K_{{\lim}}, \mu_{{\lim}})$ in which $K_{\lim}$ is a \emph{subgroup} of $K_{\chi_\infty}$;  
\item the induced representation $\ind_{K_{\lim}}^{K_{\chi_{\infty}}}(\mu_{{\lim}})$ contains $\mu_{{\infty}}$.
\end{enumerate}
 \end{theo}
 
\begin{proof}[Remark on the statement]
The result is due to Baggett \cite{Baggett}, who works with the more general setting of a semidirect product $H \ltimes N$ where $H$ is a compact group and $N$ a locally compact abelian group. His main result does not look exactly like the above: there is the anecdotal point that he needs to use nets rather than sequences, and the more serious fact that he calls in a topology (introduced by Fell \cite{Fell}) on the set  $\mathcal{A}(H)$ of pairs $(J, \tau)$, where $J$ is a closed subgroup of $H$ and $\tau$ is an irreducible representation of $J$. To see how this translates into Theorem \ref{Baggett} in our case, let us first give a statement closer in spirit to  \cite[Theorem 6.2-A]{Baggett}. For the Cartan motion group $G_0=K \ltimes \pe$, using the notations of Theorem \ref{Baggett}, Baggett proves that $\sigma_\infty$ is a limit point of $(\sigma_n)_{n \in \N}$ if and only if there is a subsequence $(\sigma_{n_j})_{j \in \N}$ such that: 
\begin{enumerate}[(a)]
\item as $j$ goes to infinity, $\chi_{n_j}$ goes to  $\chi_\infty$;
\item $(K_{\chi_{n_j}}, \mu_{n_j})$ goes, in the Fell space  $\mathcal{A}(K)$, to a pair  $(K_{\text{lim}}, \mu_{\text{lim}})$ where $K_{\lim}$ is a  \emph{subgroup of~$K_{\chi_{\infty}}$},
\item the induced representation $\inde_{K_{\lim}}^{K_{\chi_{\infty}}}(\mu_{\text{lim}})$ contains $\mu_{{\infty}}$.
\end{enumerate}
If we fix a minimal parabolic subgroup $P_{\min} = M_{\min} A_{\min} N_{\min}$, we can choose the Mackey parameters $(\chi_n, \mu_n)$ in such a way that $\chi_n$ always belongs to the closed Weyl chamber $\mathfrak{a}_{\min}^+$ that comes with $P_{\min}$. But then, since there is only a  \emph{finite} number of subgroups of $K$ that can arise as the stabilizer of  some  $\chi$ in $\mathfrak{a}_{\min}^+$, there is only a finite number of subgroups that can arise as $K_{\chi_n}$ for some $n$. In Baggett's statement above, we can thus replace (b) with
\noindent \begin{itemize}
\item[(b')] the sequence $\left((K_{\chi_{n_j}}, \mu_{n_j})\right)_{j \in \N}$ is eventually constant and $K_{\chi_{n_j}}$ is actually a subgroup of~$K_{\chi_{\infty}}$.
\end{itemize}
This leads to the  statement in Theorem \ref{Baggett}.\end{proof}

 \section{Continuity of the Mackey-Higson bijection} \label{sec:resultat}~

 \begin{theo} \label{continuite} The bijection $ \mathcal{M} \ : \ \widetilde{G} \to \widehat{G_0}$ is continuous.  \end{theo}

What we will actually check is that if $\pi_{\infty}$ is an element of $\widetilde{G}$ and if $(\pi_n)_{n \in \N}$ is a sequence of irreducible tempered representations of $G$ that admits $\pi_\infty$ as a limit point, then there exists a subsequence of $(\pi_n)_{n \in \N}$ whose image under $\mathcal{M}$ admits $\mathcal{M}(\pi_{\infty})$ as a limit point.

\subsection{Preliminaries on the reductive side}

Write  $\widetilde{G}_\Theta$ for the connected component of $\widetilde{G}$ that contains $\pi_{\infty}$. There is a subsequence of $(\pi_n)_{n \in \N}$ whose terms all lie in $\widetilde{G}_{\Theta}$; since we will have a handful of successive extractions to do hereafter, we will replace $(\pi_n)_{n \in \N}$ by such a subsequence without changing the notation. 

We now take up the setting of Notations \ref{notations} and consider a parabolic subgroup $LN$, a ``discrete series'' representation $\sigma$ of (the compactly generated part of) $L$, and parameters  $(\nu_k)_{k \in \N \cup \{\infty\}}$ in $\mathfrak{a}^+$, in such a way that  $\pi_\infty$ occurs in $\inde_{LN}^G(\sigma \otimes e^{i\nu_{\infty}})$ and  for each $n$, $\pi_n$ occurs in $\inde_{LN}^G(\sigma \otimes e^{i\nu_n})$.

By Corollary \ref{conv_param_cont}, after passing to a subsequence if necessary, we may assume that  $\nu(\pi_n)$ goes to $\nu(\pi_{\infty})$ in $\widehat{\mathfrak{a}}/W$. Since the Weyl group  $W$ is finite and $\nu_n$ (resp. $\nu_\infty$) is a representative of  $\nu(\pi_n)$ (resp. $\nu(\pi_\infty)$) in $\mathfrak{a}^+$, we may assume, perhaps after another passage to subsequences, that  $\nu_n$ in fact goes to $\nu_\infty$ in $\mathfrak{a}^+$.

\subsubsection{Strata in the Weyl chamber} For each $n$ in $\N \cup \{\infty\}$, let us consider 
\[ S_n = \left\{ \alpha \in \Delta^+ \, : \, \langle \alpha, \nu_n \rangle > 0 \right\}.\] That subset of $\Delta^+$ keeps track of the ``$\mathfrak{a}$-regularity of $\nu_n$'': when  $S_n = \emptyset$ we have $\nu_n = 0$, whereas $S_n = \Delta^+$ if and only if  $\nu_n$ is $\mathfrak{a}$-regular.

\begin{enonce}{Observation} \label{obs} Suppose $\Sigma$ is a subset of $\Delta^+$. If $\left\{ n \in \N \, : \,  S_n = \Sigma \right\}$ is infinite, then  $\Sigma$ contains $S_\infty$. \end{enonce}

This is because when $\alpha \in \Delta^+$ lies in $S_\infty$, we have $\langle \alpha, \nu_\infty \rangle >0$, so that eventually $\langle \alpha, \nu_n \rangle >0 $ and $\alpha \in S_n$.\qed\\

Let us now fix a subset $\Sigma \subset \Delta^+$ such that $\left\{ n \in \N \, : \, S_n = \Sigma \right\}$ is infinite. After passing again to a subsequence if necessary, we may assume that  $S_n = \Sigma$ for all $n$. We now observe that the centralizer\footnote{For the coadjoint action, where we view $\nu_n \in \mathfrak{a}^\star$ as an element of  $\mathfrak{g}^\star$.} $L_{\nu_n} = Z_G(\nu_n)$ of $\nu_n$ in $G$ does not depend on $n$. See for instance the details on the construction of the centralizer given in   \cite[proof of Theorem 3.2(a)]{AAMackey}: if $L_{\min} = M_{\min} A_{\min}$ is a minimal Levi subgroup of $G$ contained in $L$, then $L_{\nu_n}$ is generated by $M_{\min}$  and the root subgroups for roots  $\alpha$ that lie in $S_n = \Sigma$. These do not depend on $n$.

We will henceforth write $L_{\text{seq}}$ for the common centralizer of all $\nu_{n}$'s, $n \in \N$. Given  Observation \ref{obs}, if we write  $L_{\infty}$ for the centralizer of $\nu_{\infty}$ in $G$, then
\begin{equation} \label{inc_levi} L \subset L_{\text{seq}} \subset L_{\infty}\end{equation}
and if  $A_{\text{seq}} $ and $A_{\infty} $ denote the split components of $L_{\text{seq}}$ and $L_{\infty}$, then 
 \begin{equation} \label{inc_a} A_\infty \subset A_{\text{seq}} \subset A.\end{equation}
We remark, following the proof of \cite[Theorem 3.2]{AAMackey}, that $L$ is in fact a Levi subgroup in the reductive group $L_{\text{seq}}$, and that $L_{\text{seq}}$ is itself a Levi subgroup of $L_{\infty}$. With apologies for the clumsy notation, let us write  $\widetilde{N}$ for a subgroup of $L_{\text{seq}}$ such that $L \widetilde{N}$ is a parabolic subgroup of $L_{\text{seq}}$ with Levi factor $L$, and $N_{\text{seq}}$ for a subgroup of $L_\infty$ such that $L_{\text{seq}}N_{\text{seq}}$ is a parabolic subgroup of $L_{\infty}$ with Levi factor $L_{\text{seq}}$. We will also need the maximal compact subgroups 
\begin{equation} K_{\text{seq}} = K \cap L_{\text{seq}} \quad \text{and} \quad K_\infty = K \cap L_\infty.\end{equation}

\subsubsection{Mackey parameters} \label{jeu_ecriture} The above observations make it easy to pin down the Mackey parameters for each of the $\pi_n$s that remain at this stage (after the extractions already performed).

Recall from Vogan's work that for every irreducible representation $\mu$ of ${K_{\text{{seq}}}}$, there exists a unique irreducible tempered representation of   ${L_{\text{{seq}}}}$ that has real infinitesimal character and lowest $K_{\text{seq}}$-type $\mu$ (see \cite[Theorem 1.2]{VoganBranching}). Write $\mathbf{V}_{L_{\text{\emph{seq}}}}(\mu_n)$ for that representation. The  $K_{\text{seq}}$-type $\mu$ is the unique lowest $K_{\text{seq}}$-type of $\mathbf{V}_{L_{\text{\emph{seq}}}}(\mu_n)$. 

\begin{lemm} \label{mac} For every $n$ in $\N$, there exists $\mu_n$ in $\widehat{K_{\text{\emph{seq}}}}$ such that  $\pi_n$ is equivalent with the irreducible representation  $\ind_{P_{\text{\emph{seq}}}}^G\left(\mathbf{V}_{L_{\text{\emph{seq}}}}(\mu_n) \otimes e^{i\nu_{n}}\right)$. Furthermore, the representation $\mu_n$ is uniquely determined by the set of lowest $K$-types of  $\pi_n$. \end{lemm}

\begin{proof} Recall that $\pi_n$ occurs in $\inde_{LN}^G(\sigma \otimes e^{i\nu_n})$ and remark, as in the proof of  \cite[Theorem 3.2(c)]{AAMackey}, that if $L_{\text{seq}}N'$ is a parabolic subgroup of $G$ with Levi factor $L_{\text{seq}}$, then 
\begin{equation}\inde_{LN}^G(\sigma \otimes e^{i\nu_n}) \simeq \inde_{L_{\text{seq}}}^G\left( \inde_{L \widetilde{N}}^{L_{\text{seq}}} \left( \sigma \otimes \mathbf{1}\right) \otimes e^{i\nu_n} \right) .\end{equation}
The representation $\varpi=\inde_{L \widetilde{N}}^{L_{\text{seq}}} \left( \sigma \otimes \mathbf{1}\right)$ is tempered, has real infinitesimal character and a finite number of irreducible components. Since $\nu_n$ is $\mathfrak{a}_{\text{seq}}$-regular, for every irreducible component $\tau$ of $\varpi$, the representation $\inde_{L_{\text{seq}}N}^G(\tau \otimes e^{i\nu_n})$ is in fact irreducible (see for instance  \cite[Lemma 3.2(5)]{VoganLanglands} and the discussion in \cite[p. 8]{AAMackey}). Furthermore, its lowest $K$-types are entirely determined by $\tau$. Thus, the lowest $K$-types of $\pi_n$ make it possible to pin down the one irreducible constituent  $\tau_n$ for which $\pi_n  \simeq \inde_{L_{\text{seq}}N}^G(\tau_n \otimes e^{i\nu_n})$. Writing $\mu_n$ for the unique lowest $K_{\text{seq}}$-type of $\tau_n$, we obtain $\tau_n \simeq \mathbf{V}_{L_{\text{\emph{seq}}}}(\mu_n)$, as desired.  \end{proof}

In the above statement, $\mu_n$ may depend on $n$; but it determines the set of lowest $K$-types of $\pi_n$, which is a subset of the class $\mathcal{C}_\Theta$ described in \S \ref{vogan}, and different values for $\mu_n$ lead to different sets of lowest $K$-types \cite[Lemma 4.2]{AAMackey}. So there are only a finite number of possibilities for the element $\mu_n$ in $\widehat{K_{\text{seq}}}$. After a new extraction, we obtain the following result:   

\begin{lemm} There exists a subsequence $(\pi_{n_j})_{j \in \N}$ of $(\pi_n)_{n \in \N}$ with the property that $\mu_{n_j}$ does not depend on $j$, so that there exists $\mu_{\text{\emph{seq}}} \in \widehat{K_{\text{\emph{seq}}}}$ such that: $\forall j \in \N$, $\pi_{n_j} \simeq \ind_{P_{\text{\emph{seq}}}}^G\left(\mathbf{V}_{L_{\text{\emph{seq}}}}(\mu_{\text{\emph{seq}}}) \otimes e^{i\nu_n}\right)$. \end{lemm}

If we replace $(\pi_n)_{n \in \N}$ by the above subsequence and take up the notations of \S \ref{param_mackey}, we now know from Lemma \ref{lemme:mackey} that for each $n$, a Mackey parameter for (the new) $\pi_n$ is $(\nu_n, \mu_\text{seq})$.

\subsubsection{Remark on the lowest  $K$-types of $\pi_{\infty}$.} At this stage, since each $\nu_n$ is $\mathfrak{a}_{\text{seq}}$-regular, we know that the set of lowest $K$-types of   $\pi_n$ does not depend on $n$: it is determined by the pair  $(K_{\text{seq}}, \mu_{\text{seq}})$.  Example \ref{ex_lim} shows that $\pi_\infty$ does not necessarily have the exact same set of lowest $K$-types as the $\pi_n$s: if $G = \SL(2,\R)$ and $\pi_\infty$ is a limit of discrete series, then $\pi_\infty$ has a unique lowest $\SO(2)$-type, but  $\pi_\infty$ is a limit point of the nonspherical principal series, which consists of representations having two distinct lowest $\SO(2)$-types (one of which is that of $\pi_{\infty}$).

\begin{lemm} \label{ktype_commun} Suppose $\mathcal{C}_{\text{\emph{seq}}} \subset \widehat{K}$ is the set of lowest $K$-types common to all $\pi_n$, $n \in \N$. Then the set of  lowest $K$-types of  $\pi_\infty$ is contained in $\mathcal{C}_{\text{\emph{seq}}}$. \end{lemm} 
\begin{proof} We first recall that if $\lambda$ is a class in $\widehat{K}$ and $\chi_\lambda: K \to \C$ is its character, there is a simple criterion for ascertaining that an irreducible tempered representation  $\pi$ of $G$  does \emph{not} contain $\lambda$ upon restriction to $K$: one needs only check that for every matrix element  $c: G \to \C$ of $\pi$, the convolution product\footnote{This is  the convolution product over $K$: the function $k \mapsto  \int_K \chi_\lambda(ku^{-1})c(u) du$ from $K$ to $\C$.
} $\chi_\lambda \star \left( c_{|K}\right)$ vanishes.

We also recall that the partial ordering on $\widehat{K}$ used to define the notion of lowest $K$-type can be built from a positive-valued function $\norme{\cdot}_{\widehat{K}}$ on $\widehat{K}$. All  classes $\lambda$ in $\mathcal{C}_\text{seq}$ have the same norm $\norme{\lambda}_{\widehat{K}}$ ; write $\norme{\mathcal{C}_{\text{seq}}}_{\widehat{K}}$ for the common value. 

Consider then
\begin{itemize}
\item[$\bullet$]  a $K$-type $\lambda$ such that $\norme{\lambda}_{\widehat{K}} \leq \norme{\mathcal{C}_{\text{seq}}}_{\widehat{K}}$, but which does not belong to $\mathcal{C}_{\text{seq}}$,
\item[$\bullet$] and a matrix element $c$ of $\pi_\infty$;  \end{itemize}
let us inspect the convolution $\chi_\lambda \star \left( c_{|K}\right)$. By definition of the Fell topology, there exists a sequence  $(c_n)_{n \in \N}$ of complex-valued functions on $G$, in which each map $c_n$ is a matrix element of $\pi_n$, and which as $n$ goes to infinity goes to $c$ uniformly on compact subsets of $G$. Since the $K$-type $\lambda$ appears in none of the $\pi_n$, we have $\chi_\lambda \star \left( (c_{n})_{|K}\right) = 0$ for each $n$. But  $((c_n)_{|K})_{n \in \N}$ goes to $c_{|K}$ uniformly on $K$; so $\chi_\lambda \star \left( (c_{n})_{|K}\right) $ does pointwise converge to the  $\chi_\lambda \star\left( c_{|K}\right)$. Thus the latter function must be zero. 

We conclude that a   $K$-type whose norm does not exceed $\norme{\mathcal{C}_{\text{seq}}}_{\widehat{K}}$ can appear in $\pi_\infty$ only if it belongs to  $\mathcal{C}_{\text{seq}}$. Besides, $\pi_\infty$ lies in $\widetilde{G}_\Theta$, so we already know (by Proposition \ref{vogan}) that its lowest $K$-types all have norm $\norme{\mathcal{C}_{\text{seq}}}_{\widehat{K}}$; this proves the lemma.
\end{proof}

\subsection{Verification of Baggett's criterion}

We can now prove that  $\mathcal{M}(\pi_\infty)$ is a limit point of $(\mathcal{M}(\pi_n))_{n \in \N}$ in the unitary dual  $\widehat{G_0}$. In \S \ref{jeu_ecriture}, we showed that a Mackey parameter for $\pi_n$ is $(\nu_n, \mu_\text{seq})$ $-$ recall that in the present context, the centralizer of  $\nu_n$ in $K$ is equal to $K_\text{seq}$ (and independent of $n$).

By the argument already used for  Lemma \ref{mac}, we also know that there exists a parabolic subgroup $P_\infty$ with Levi factor $L_\infty$, and an irreducible representation $\mu_\infty \in \widehat{K_\infty}$, such that  $\pi_\infty$ is equivalent with the irreducible representation $\inde_{P_\infty}^G(\mathbf{V}_{L_\infty}(\mu_\infty) \otimes e^{i\nu_\infty})$; the representation  $\pi_\infty$ then admits  $(\nu_\infty, \mu_\infty)$ as a Mackey parameter.

Given Baggett's criterion \ref{Baggett}, to complete the proof of Theorem \ref{continuite}, we need only verify the following fact:

\begin{lemm} The representation $\mu_{\infty}$ of $K_\infty$ is contained in  $\ind_{K_{\text{\emph{seq}}}}^{K_{\infty}}(\mu_{\text{\emph{seq}}})$. \end{lemm}

\begin{proof} Suppose the Lemma is false. Let us induce to $K$ and compare the $K$-modules 
\begin{equation} \label{restr} \inde_{K_\infty}^K(\mu_\infty) \quad\quad  \text{and}  \quad\quad \inde_{K_\infty}^K\left( \inde_{K_{\text{{seq}}}}^{K_{\infty}}(\mu_{\text{{seq}}}) \right) \simeq \inde_{K_{\text{{seq}}}}^{K}(\mu_{\text{{seq}}}) ; \end{equation}
more precisely, let us inspect their lowest $K$-types. If $\tilde{\mu}$ is an irreducible representation of $K_{\infty}$ that appears in $\inde_{K_{\text{{seq}}}}^{K_{\infty}}(\mu_{\text{{seq}}})$, then under our assumption that the Lemma is false, we have $\mu_\infty \neq \tilde{\mu}$; but in that situation, we know \cite[Lemma 4.2]{AAMackey} that the representations  $\inde_{K_\infty}^K(\mu_\infty)$ and $\inde_{K_\infty}^K(\tilde{\mu})$ have no lowest $K$-type in common. The second $K$-module in \eqref{restr} is a direct sum of  $K$-modules that each read $\inde_{K_\infty}^K(\tilde{\mu})$ for some $\tilde{\mu} \neq \mu_\infty$; so the two $K$-modules to be compared in \eqref{restr} actually have no lowest $K$-type in common.

To see that this is impossible, we only have to point out that the two $K$-modules in \eqref{restr} are the restrictions to $K$ of $\pi_\infty$ and $\pi_n$ (see \cite{AAMackey}, Remark 2.3). Lemma \ref{ktype_commun} shows that each lowest $K$-type in  $\pi_\infty$ is also a lowest $K$-type in $\pi_n$, so the set of lowest $K$-types of the first module of  \eqref{restr} is contained in the set of lowest $K$-types of the second. The Lemma follows. 
  \end{proof}

\backmatter

\bibliography{topo}
\bibliographystyle{smfplain}

\end{document}